\documentclass[hidelinks,onefignum,onetabnum]{siamart250211}
\usepackage{lipsum}
\usepackage{amsfonts}
\usepackage{graphicx}
\usepackage{epstopdf}
\ifpdf
  \DeclareGraphicsExtensions{.eps,.pdf,.png,.jpg}
\else
  \DeclareGraphicsExtensions{.eps}
\fi

\newsiamremark{remark}{Remark}
\newsiamremark{hypothesis}{Hypothesis}
\crefname{hypothesis}{Hypothesis}{Hypotheses}
\newsiamthm{claim}{Claim}
\newsiamremark{fact}{Fact}
\crefname{fact}{Fact}{Facts}

\headers{Structure-Exploiting IMEX-TR for the LB Model}{Chenglong Bao, Kai Deng, Kai Jiang, Juan Zhang}

\title{A Structure-Exploiting Implicit-Explicit Trust Region Method for Computing Second-Order Stationary Points of the Landau-Brazovskii Model\thanks{Submitted to the editors DATE.
\funding{}}}

\author{
Chenglong Bao\thanks{Yau Mathematical Sciences Center, Tsinghua University, Beijing, 100084, China (\email{clbao@mail.tsinghua.edu.cn}).}
\and Kai Deng\thanks{School of Mathematics and Computational Science, Hunan Key Laboratory for Computation and Simulation in Science and Engineering, Key Laboratory of Intelligent Computing and Information Processing of Ministry of Education, Xiangtan University, Xiangtan, Hunan, 411105, China (\email{kdeng1533@smail.xtu.edu.cn, kaijiang@xtu.edu.cn, zhangjuan@xtu.edu.cn}).}
\and Kai Jiang\footnotemark[3]~\thanks{Corresponding author.} \and Juan Zhang\footnotemark[3]
}

\usepackage{amsopn}

\usepackage{bm}%
\usepackage{float}%
\usepackage{braket}%
\usepackage{lipsum}%
\usepackage{amssymb}%
\usepackage{amsmath}%
\usepackage{amsfonts}%
\usepackage{graphicx}%
\usepackage{subfigure}%
\usepackage{epstopdf}%
\usepackage{mathrsfs}%
\usepackage{multicol}%
\usepackage{multirow}%
\usepackage{rotating}
\usepackage{algorithm}%
\usepackage{algorithmicx}%
\usepackage{algpseudocode}%
\usepackage{enumitem}%
\usepackage{xcolor}%
\usepackage{placeins}%
\usepackage{makecell}%

\DeclareMathOperator*{\argmin}{\arg\min}

\ifpdf
\hypersetup{
  pdftitle={Discovering new phases via computing second-order stationary states of Landau-Brazovskii model},
  pdfauthor={Chenglong Bao, Kai Deng, Kai Jiang, Juan Zhang}
}
\fi

\begin{document}

\maketitle

\begin{abstract}
This work focuses on the reliable computation of second-order stationary points in the high-dimensional nonconvex energy landscape of the Landau-Brazovskii (LB) model, a fundamental model for studying phases and phase transitions.
For this purpose, we develop an efficient implicit-explicit trust region (IMEX-TR) method.
Trust region (TR) methods can avoid saddle-point stagnation and guarantee convergence to second-order stationary points under appropriate conditions.
However, their direct application to the LB model has been impractical because the Hessian is dense if treated directly.
The proposed IMEX-TR method overcomes this difficulty by exploiting the Hessian's special structure: the linear interaction part is diagonal in reciprocal space, whereas the nonlinear bulk-energy part is diagonal in physical space.
Based on this structure, we design an efficient solver for the TR subproblem that with globally convergent guarantee and enjoys FFT-based acceleration, with $\mathcal O(N \log N)$ complexity per iteration.
Existing first-order gradient-based methods for the LB model only guarantee convergence to first-order stationary points and may stagnate at saddle points.
In contrast, the proposed IMEX-TR method inherits the theoretical guarantee of converging to second-order stationary points while remaining computationally practical.
Numerical experiments verify the theoretical properties of the algorithm and demonstrate its robustness in locating stable phases from different initial conditions.
Numerical results also show that IMEX-TR can escape unstable stationary states reached by first-order schemes and converge to physically meaningful second-order stationary points.
These results suggest that targeting second-order stationary points provides an effective computational paradigm for exploring complex free-energy landscapes and identifying stable or metastable states.
\end{abstract}



\section{Introduction}\label{sec:introd}
Computing stable and metastable states of Landau energy functionals is a fundamental task in scientific computing, because such states govern equilibrium phases and phase transitions in ordered materials.
The Landau-Brazovskii (LB) energy functional is a prototypical model for studying weak crystallization and modulated phases, capturing the balance between short-range segregation and long-range frustration in systems from block copolymers to magnetic skyrmion lattices~\cite{brazovskiui1996phase, swift1977hydrodynamic, caplan2017}.
Over the past decades, extensive numerical studies have shown that the LB model possesses a complex, high-dimensional nonconvex free-energy landscape with many competing stationary points, including local minima and saddle points~\cite{mcclenagan2019landau}.
A central numerical difficulty lies in ensuring that the computed stationary states are truly local minima rather than saddle points, thereby benefiting phase diagram construction.

A major obstacle to finding ordered patterns and constructing phase diagram is the problem of getting stuck at saddle points (SDPs) during numerical simulations.
Standard computational approaches, explicit-implicit schemes~\cite{wise2009energy, chen1998applications, shen2010numerical,li2016characterizing}, scalar auxiliary variable schemes~\cite{shen2019new}, the invariant energy quadrature scheme~\cite{yang2016linear}, exponential time differencing~\cite{fu2022energy,hou2023etd} and the Bregman proximal gradient approaches~\cite{bao2024convergence,bao2020adaptive,jiang2020efficient} are designed to lower the free energy at each step.
These approaches are widely used, but they are inherently limited to first-order stationary points, because they terminate as soon as the gradient of the free energy vanishes.
In the high-dimensional, non-convex energy landscape of the LB model, this condition is satisfied not only by stable or metastable but also by many transition states (SDPs with negative eigenvalues of the Hessian).
As a result, a simulation starting near a phase boundary might incorrectly converge to an unstable disordered state or a saddle point, thereby obscuring the stable states that should be used for phase-diagram construction.

To reliably identify (meta-)stable states, one must work with a sharper criterion that distinguishes genuine local minima from saddle points.
This criterion is precisely captured by the notion of second-order stationary points.

\begin{definition}[Second-order stationary points]
For a twice differentiable function $E:\mathbb{R}^n\rightarrow \mathbb{R}$, a point $\bm{v}\in\mathbb{R}^n$ is called a first-order stationary point (SP-I) of $E$ if it satisfies $\nabla E(\bm{v}) = 0$. Moreover, if $\bm{v}$ is an SP-I and satisfies $\nabla^2 E(\bm{v}) \succeq 0$, then $\bm{v}$ is called a second-order stationary point (SP-II) of $E$.
\end{definition}

SP-Is encompass global or local minima, local minima, and saddle points (SDPs), corresponding to stable, metastable and transition states in physics, respectively.
SP-IIs form a subset of SP-Is and represent physically (meta-)stable states.
The identification of SP-IIs enables accurate determination of equilibrium states.
We define SP-IIs using positive semi-definite Hessians, rather than strictly positive definite ones, to accommodate degenerate stationary states with zero eigenvalues in their Hessian matrices.
Such degeneracy arises fundamentally from continuous symmetries, most notably translational invariance in ordered phases including periodic crystals and quasicrystals~\cite{bao2024convergence,cui2025efficient,cui2024spring}. 

Several SDPs search methods have been developed to locate SDPs and transition pathways~\cite{cui2025efficient,yin2019high,henkelman1999dimer,e2011gentlest}. By computing or tracking unstable Hessian eigendirections, these methods can distinguish SDPs from local minima and recover neighboring SP-IIs through downward searches from the saddles. They therefore provide a systematic route to identifying SP-IIs and characterizing the connectivity of the energy landscape. However, repeatedly approximating unstable eigendirections can be computationally expensive for high-dimensional discretizations of the LB model.

Trust region (TR) methods are a classical tool in nonconvex optimization and are well known for exploiting curvature information to avoid saddle point stagnation. Moreover, modern TR theory (e.g., TRACE~\cite{curtis2017trust}) establishes convergence to SP-IIs under appropriate conditions. However, applying such methods to the LB model is challenging because of the high-dimensional discretisation and the dense coupling in the Hessian. In this work, we overcome these challenges by designing a structure-exploiting TR method that makes these TR advantages practically attainable for the LB model.

The main contributions of this paper are as follows.
\begin{itemize}
\item We propose the first efficient implicit-explicit trust region (IMEX-TR) method for computing SP-IIs of the LB model. While saddle avoidance through negative curvature is standard in the TR literature and SP-II convergence follows from modern TR theory under appropriate assumptions, direct application of TR to the LB model is computationally prohibitive.
\item We develop an efficient solver for the nonconvex TR subproblem arising at each iteration. The solver identifies negative-curvature directions and computes the global minimizer. In particular, we exploit the special structure of the LB Hessian: the linear interaction part is diagonal in reciprocal space, whereas the nonlinear bulk-energy part is diagonal in physical space. This structure enables FFT-based acceleration and yields $\mathcal{O}(N \log N)$ complexity per iteration.
\item We establish the convergence theory of the IMEX-TR algorithm by combining the TRACE convergence framework with the global optimality of the inner subproblem solver.
\item Numerical experiments validate the theory and show that IMEX-TR consistently converges to SP-IIs. They also demonstrate clear advantages over first-order gradient-based methods, which may stagnate at saddle points or other unstable stationary states.
\end{itemize}

The remainder of this article is organized as follows. Section~\ref{sec:discretization} discretizes the LB energy functional via the Fourier pseudospectral method, obtaining the finite-dimensional optimization problem. Section~\ref{sec:tr_method} proposes the IMEX-TR method and provides a comprehensive convergence analysis to SP-IIs. This section also details the construction and theoretical analysis of the efficient global solver for the nonconvex TR subproblem. Section~\ref{sec:numerical} presents numerical results that validate the theoretical properties of the proposed method. Finally, Section~\ref{sec:conclusion} provides a summary of the main findings.

\textbf{Notation}: In the rest of the content, let $\|\bm{u}\|_2 = \sqrt{\bm{u}^\top \bm{u}}$ denote the $\ell_2$ norm for vector $\bm{u}$. The operator (spectral) norm of a matrix $\bm{H}$ is denoted by $\|\bm{H}\| := \max\limits_{\|\bm{u}\|_2=1} \|\bm{H}\bm{u}\|_2$. For a third-order tensor $\mathcal{T}$ (e.g., third-order derivative $\nabla^3 E$), the operator norm is defined as $\|\mathcal{T}\| := \max\limits_{\|\bm{u}\|_2 = \|\bm{v}\|_2 = \|\bm{w}\|_2 = 1} |\mathcal{T}[\bm{u},\bm{v},\bm{w}]|$, which represents the maximum absolute value of the tensor evaluated over all unit vectors \cite{cartis2020sharp}.

\section{Discretizing the LB model with Fourier pseudospectral method}\label{sec:discretization}

The energy functional of the LB model is
\begin{equation}\label{eq:lb_energy}
E(\psi(\bm{x})) = \frac{1}{|\Omega|} \int_{\Omega}\underbrace{\frac{1}{2} [(\Delta + 1)\psi(\bm{x})]^2 }_{G(\psi(\bm{x}))} + \underbrace{ (\frac{\tau}{2!} \psi(\bm{x})^2 - \frac{\gamma}{3!} \psi(\bm{x})^3 + \frac{1}{4!} \psi(\bm{x})^4 ) }_{F(\psi(\bm{x}))} d\bm{x}.
\end{equation}
Here, $\Omega$ is a bounded domain with volume $|\Omega|$. The parameter $\tau$ denotes the dimensionless reduced temperature, and $\gamma$ characterizes the strength of the three-body interaction. The order parameter $\psi(\bm{x})$ satisfies the mass conservation constraint 
$$\frac{1}{|\Omega|}\int_{\Omega}\psi(\bm{x})d\bm{x} = 0.
$$

We consider the periodic structures of the LB model. Thus, we adopt the Fourier pseudospectral method for spatial discretization. This approach naturally accommodates periodic boundary conditions and facilitates the enforcement of mass conservation.

Consider a periodic order parameter $\psi(\bm{x})$ defined on $\bm{x}\in\mathbb{T}^d:=\mathbb{R}^d/\bm{A}\mathbb{Z}^d$, where $\bm{A}\in\mathbb{R}^{d\times d}$ represents the primitive Bravais lattice. The primitive reciprocal lattice $\bm{B}\in\mathbb{R}^{d\times d}$ satisfies the dual relationship $\bm{A}\bm{B}^{\top}=2\pi \bm{I}$. The order parameter $\psi(\bm{x})$ can be expanded as
\[
  \psi(\bm{x}) = \sum_{\bm{k}\in\mathbb{Z}^d}\hat{\psi}(\bm{k})e^{i(\bm{B}\bm{k})^{\top}\bm{x} },\quad \bm{x}\in \mathbb{T}^d,
\]
where the Fourier coefficients $\hat{\psi}(\bm{k}) = \frac{1}{|\mathbb{T}^d|}\int_{\Omega}\psi(\bm{x}) e^{-i(\bm{B}\bm{k})^{\top}\bm{x}}\,d\bm{x}$.

Define the discrete grid set as
\[
  \mathbb{T}^d_M = \left\{(\bm{x}_{1,i_1},\dots,\bm{x}_{d,i_d})=\bm{A}(i_1/M,\dots,i_d/M)^{\top},i_j=0,\dots,M-1,j=1,\dots,d\right\},
\]
where the number of elements in $\mathbb{T}^d_M$ is $N=M^d$. Denote the grid periodic function space $\mathcal{F}_M=\{f:\mathbb{T}^d_M\rightarrow \mathbb{R}, f\ \text{is periodic}\}$. For any periodic grid functions $f,g\in\mathcal{F}_M$, define the discrete $\ell^2$-inner product as $\braket{f,g}_N=\frac{1}{N}\sum_{\bm{x}_i\in\mathbb{T}^d_M}f(\bm{x}_i)\bar{g}(\bm{x}_i)$,
where $\bar{g}(\bm{x}_i)$ is the conjugate of $g(\bm{x}_i)$.

The discrete reciprocal space is
\[
  \mathbb{K}_M^d=\{\bm{k}=(k_i)_{i=1}^d\in\mathbb{Z}^d:-M/2\leq k_i<M/2\}.
\]
For $\bm{k}, \bm{h}\in\mathbb{Z}^d$, we have the discrete orthogonality
\[\braket{ e^{i(\bm{B}\bm{k})^{\top}\bm{x}},~ e^{i(\bm{B}\bm{h})^{\top}\bm{x}} }_N
=\begin{cases}1,& \bm{k}=\bm{h}+M\bm{m},~ \bm{m}\in\mathbb{Z}^d,\\
0,& \mbox{otherwise}.
\end{cases}\]
Then the discrete Fourier coefficients of $\psi(\bm{x})$ in $\Omega_M$ can be represented as
\[\hat{\psi}(\bm{k})=\braket{\psi(\bm{x}), e^{i(\bm{B}\bm{k})^{\top}\bm{x}} }_N=\frac{1}{N}\sum_{\bm{x}_i\in\mathbb{T}^d_M}\psi(\bm{x}_i)e^{-i(\bm{B}\bm{k})^{\top}\bm{x}_i},\quad \bm{k}\in \mathbb{K}_M^d.\]

Let $\bm{\Psi}=(\psi_1,\dots,\psi_N)\in\mathbb{R}^{N}$, where $\psi_i=\psi(\bm{x}_i),1\leq i\leq N$. Denote $\mathscr{F}_N\in \mathbb{C}^{N\times N}$ as the discrete Fourier transformation matrix, then we have the discrete order parameter $\hat{\bm{\Psi}}=\mathscr{F}_N \bm{\Psi}$ in reciprocal space. The discretized energy function is 
\[E(\hat{\bm{\Psi}}) = G(\hat{\bm{\Psi}}) + F(\hat{\bm{\Psi}}),\]
where $G$ and $F$ are the discretized interaction and bulk energy
\begin{equation}\label{eq:interaction_bulk_energy}
\begin{aligned}
G(\hat{\bm{\Psi}} ) &= \frac{1}{2}
\sum_{\bm{k}_1+\bm{k}_2=0} \bigl[1 - \bm{k}^{\top}\bm{k}\bigr]^2\hat{\bm{\Psi}} (\bm{k}_1)\hat{\bm{\Psi}}(\bm{k}_2),\\
F(\hat{\bm{\Psi}}) &= \frac{\tau}{2}\sum\limits_{\bm{k}_1+\bm{k}_2=0}\hat{\bm{\Psi}} (\bm{k}_1)\hat{\bm{\Psi}}(\bm{k}_2) - \frac{\gamma}{3!}\sum\limits_{\bm{k}_1+\bm{k}_2+\bm{k}_3=0}\hat{\bm{\Psi}}(\bm{k}_1)\hat{\bm{\Psi}}(\bm{k}_2)\hat{\bm{\Psi}}(\bm{k}_3)\\
&+\frac{1}{4!}\sum\limits_{\bm{k}_1+\bm{k}_2+\bm{k}_3+\bm{k}_4=0}\hat{\bm{\Psi}}(\bm{k}_1)\hat{\bm{\Psi}}(\bm{k}_2)\hat{\bm{\Psi}}(\bm{k}_3)\hat{\bm{\Psi}}(\bm{k}_4),
\end{aligned}
\end{equation}
and $\bm{k}_i\in\mathbb{K}_M^d$. The nonlinear terms $F(\hat{\bm{\Psi}})$ can be efficiently obtained by the FFT after being calculated in the $d$-dimensional physical space. Moreover, the mass conservation constraint is discretized as $\bm{e}_1^{\top}\hat{\bm{\Psi}} = 0$,
where $\bm{e}_1 = (1, 0, \dots , 0)^{\top} \in \mathbb{R}^N$. Therefore, we obtain the finite-dimensional minimization problem
\begin{equation}\label{eq:constrained_ener_one_parameter}
  \min_{\hat{\bm{\Psi}} \in \mathbb{C}^N} E(\hat{\bm{\Psi}}) = G(\hat{\bm{\Psi}} ) + F(\hat{\bm{\Psi}}) \quad s.t.\ \bm{e}_1^{\top}\hat{\bm{\Psi}} = 0.
\end{equation}
According to \eqref{eq:interaction_bulk_energy}, we have
\[
\nabla G(\hat{\bm{\Psi}}) = \bm{\Lambda}\hat{\bm{\Psi}},\quad \nabla F(\hat{\bm{\Psi}}) = \mathscr{F}_N\bm{\Gamma} \mathscr{F}_N^{-1} \hat{\bm{\Psi}},
\]
\[
\nabla^2G(\hat{\bm{\Psi}}) = \bm{\Lambda},\quad \nabla^2F(\hat{\bm{\Psi}}) = \mathscr{F}_N\bm{\Gamma}^{\prime}\mathscr{F}_N^{-1},
\]
where $\bm{\Lambda}$ is a diagonal matrix with non-negative entries $\bigl[1 - \bm{k}^{\top}\bm{k}\bigr]^2$ and $\bm{\Gamma} , \bm{\Gamma}^{\prime}\in \mathbb{R}^{N\times N}$ are also diagonal matrices that depend on $\hat{\bm{\Psi}}$.

\section{Implicit-Explicit Trust Region Method}\label{sec:tr_method}
This section considers solving the minimization problem \eqref{eq:constrained_ener_one_parameter}. To transform the problem into an unconstrained optimization problem, we use the following property
\[
\{\hat{\bm{\Psi}} \in \mathbb{C}^{N} : \bm{e}_1^{\top} \hat{\bm{\Psi}} = 0 \} = \{ \bm{P}^{\top} \bm{v} : \bm{v} \in \mathbb{C}^{N-1} \},
\]
where $\bm{P} = [\bm{0}, \bm{I}_{N-1}]$. Based on this, we introduce an equivalent objective function
\[
\bar{E}(\bm{v}) = E(\bm{P}^{\top}\bm{v}),
\]
we then can equivalently express problem \eqref{eq:constrained_ener_one_parameter} as the following unconstrained optimization problem
\begin{equation}\label{eq:transferd_nonconstrainted_minimization}
\min_{\bm{v} \in \mathbb{C}^{N-1}} \bar{E}(\bm{v}).
\end{equation}

\subsection{IMEX-TR Method and Convergence Analysis}
In this subsection, we present the IMEX-TR method (see Algorithm~\ref{alg:unified_imex_tr}) and prove that this algorithm converges to SP-IIs of the problem \eqref{eq:transferd_nonconstrainted_minimization}. The outer iteration follows the TR framework TRACE proposed in~\cite{curtis2017trust}, for which convergence to SP-IIs has been established under appropriate conditions. Our contributions are twofold: (i) an efficient global solver for the TR subproblem (Section~\ref{subsec:subproblem_solver}) that exploits the Hessian structure; and (ii) a rigorous convergence guarantee (Theorem~\ref{thm:global_optimality}) establishing that the proposed algorithm converges to the global minimizer of TR subproblem, which is important for TRACE to converge to SP-IIs.

At each iteration $j$, we construct a quadratic model $m_j(\bm{d})$ of the objective function around the current iterate $\bm{v}_j$,
\[
m_j(\bm{d}) = \bar{E}(\bm{v}_j) + \bm{g}_j^{\top} \bm{d} + \frac{1}{2} \bm{d}^{\top} \bm{H}_j \bm{d},
\]
where $\bm{g}_j = \nabla \bar{E}(\bm{v}_j)$ and $\bm{H}_j = \nabla^2 \bar{E}(\bm{v}_j)$. The step $\bm{d}_j$ is obtained by solving the TR subproblem
\begin{equation}\label{eq:tr_subproblem_with_subscript}
\min_{\|\bm{d}\|_2 \leq r_j} m_j(\bm{d}),
\end{equation}
where $r_j$ is the trust region radius. The optimality conditions for the TR subproblem \eqref{eq:tr_subproblem_with_subscript} are stated in the following theorem.
\begin{theorem}[{\cite{conn2000trust}, Corollary 7.2.2}]
A vector $\bm{d}_*$ is a global minimizer of the subproblem \eqref{eq:tr_subproblem_with_subscript} if and only if there exists a scalar $\lambda_*\geq 0$ such that the following conditions are satisfied.
\begin{equation}\label{eq:tr_subproblem_optimal_condition}
\begin{cases}
\bm{g}_j + (\bm{H}_j + \lambda_*\bm{I}) \bm{d}_* = 0,\\
\|\bm{d}_*\|_2 - r_j \leq 0,\\
\lambda_*(\|\bm{d}_*\|_2 - r_j) = 0,\\
\bm{H}_j + \lambda_* \bm{I} \succeq 0.
\end{cases}
\end{equation}
If $\bm{H}_j + \lambda_* \bm{I}$ is positive definite, then $\bm{d}_*$ is the unique global minimizer of \eqref{eq:tr_subproblem_with_subscript}.
\end{theorem}

Algorithm~\ref{alg:unified_imex_tr} summarizes the IMEX-TR method for solving the problem \eqref{eq:transferd_nonconstrainted_minimization}. Algorithm~\ref{alg:contract}, the CONTRACT procedure, is designed to adjust the trust region radius. Algorithm~\ref{alg:adaptive-implicit-explicit} in the next subsection will present the customized solver for the TR subproblem \eqref{eq:tr_subproblem_with_subscript}.

\begin{algorithm}[htbp]
\caption{IMEX-TR method for solving the problem \eqref{eq:transferd_nonconstrainted_minimization}}\label{alg:unified_imex_tr}
\begin{algorithmic}[1]
\Require Initial guess $\bm{v}_0$, initial radius $r_0$, parameters $\theta \in (0,1)$, $0<\gamma_C<1<\gamma_E$, $\gamma_\lambda > 1$, $0<\underline{\mu}\leq \bar{\mu}$, $\mu_0\geq \underline{\mu}$, $0 < r_0 \leq \nu_0$, tolerance $\varepsilon>0$.
\For{$j = 0, 1, \dots $}
    \State Compute $\bm{g}_j = \nabla \bar{E}(\bm{v}_j)$
    \State \textbf{if}\ {$\|\bm{g}_{j}\|_\infty < \varepsilon$}, \Return $\bm{v}_{j}$. \textbf{end if}
    \State Solve the TR subproblem \eqref{eq:tr_subproblem_with_subscript} (see Algorithm~\ref{alg:adaptive-implicit-explicit}), let $\bm{s}_j$ and $\lambda_j$ satisfy the optimality conditions~\eqref{eq:tr_subproblem_optimal_condition}, and $\bm{s}_j$ be the global minimizer.
    \If {$j>0$ and $\varrho_{j-1} < \theta$}
        \State $\mu_j = \max\{\mu_{j-1},\lambda_j/\|\bm{s}_j\|_2\}$
    \EndIf
    \State Calculate $\varrho_j = (\bar{E}(\bm{v}_j) - \bar{E}(\bm{v}_j + \bm{s}_j))/\|\bm{s}_j\|_2^3$
    \If{$\varrho_j \geq \theta$ and either $\lambda_j \leq \mu_j\|\bm{s}_j\|_2$ or $\|\bm{s}_j\|_2 = \nu_j$}
        \State $\bm{v}_{j+1} = \bm{v}_j + \bm{s}_j$,\ $\nu_{j+1} = \max\{\nu_j, \gamma_E\|\bm{s}_j\|_2\}$
        \State $r_{j+1} = \min\{\nu_{j+1}, \max\{r_j,\gamma_E\|\bm{s}_j\|_2\}\}$,\ $\mu_{j+1} = \max\{\mu_{j}, \lambda_j/\|\bm{s}_j\|_2\}$
    \ElsIf{$\varrho_j < \theta$}
        \State $\bm{v}_{j+1} = \bm{v}_j$,\ $\nu_{j+1} = \nu_j$,\ $r_{j+1} = \text{CONTRACT}(\bm{v}_j,r_j,\mu_j,\bm{s}_j,\lambda_j)$.
    \Else
        \State $\bm{v}_{j+1} = \bm{v}_j$,\ $\nu_{j+1} = \nu_j$,\ $r_{j+1} = \min\{\nu_{j+1},\lambda_j/\mu_j\}$,\ $\mu_{j+1} = \mu_j$
    \EndIf
\EndFor
\end{algorithmic}
\end{algorithm}

\begin{algorithm}[htbp]
\caption{CONTRACT procedure}\label{alg:contract}
\begin{algorithmic}[1]
\State \textbf{procedure} CONTRACT($\bm{v}_j$, $r_j$, $\mu_j$, $\bm{s}_j$, $\lambda_j$)
    \If{$\lambda_j < \underline{\mu}\|\bm{s}_j\|_2$}
        \State $\hat{\lambda} = \lambda_j + (\underline{\mu}\|\bm{g}_j\|_2)^{1/2}$
        \State $\tilde{\lambda}=\hat{\lambda}$, solve the linear system $(\bm{H}_j +\tilde{\lambda}\bm{I})\bm{s}=-\bm{g}_j$ to get $\bm{s}$
        \If{$\tilde{\lambda}/\|\bm{s}\|_2\leq \bar{\mu}$}
            \State \Return $r_{j+1} = \|\bm{s}\|_2$
        \Else
            \State Find $\tilde{\lambda}\in (\lambda_j,\hat{\lambda})$ such that $\underline{\mu}\leq \tilde{\lambda}/\|\bm{s}\|_2 \leq \bar{\mu}$, where $\bm{s}$ is the solution of the linear system $(\bm{H}_j +\tilde{\lambda}\bm{I})\bm{s}=-\bm{g}_j$
            \State \Return $r_{j+1} = \|\bm{s}\|_2$
        \EndIf
    \Else 
        \State $\tilde{\lambda} = \gamma_{\lambda}\lambda_j$ and solve the linear system $(\bm{H}_j +\tilde{\lambda}\bm{I})\bm{s}=-\bm{g}_j$ to get $\bm{s}$
        \If{$\|\bm{s}\|_2 \geq \gamma_C \|\bm{s}_j\|_2$}
            \State \Return $r_{j+1} = \|\bm{s}\|_2$
        \Else
            \State \Return $r_{j+1} = \gamma_C\|\bm{s}\|_2$
        \EndIf
    \EndIf
\State \textbf{end procedure}
\end{algorithmic}
\end{algorithm}

\subsubsection{Convergence to Second-Order Stationary Points}
We first present two important properties of the gradient and Hessian of the energy functional $\bar{E}(\bm{v})$ with $\{\bm{v}_j\}$ and $\{\bm{s}_j\}$ generated by Algorithm~\ref{alg:unified_imex_tr}. These properties are assumptions in the TRACE theory~\cite{curtis2017trust} and are verified in the following theorem.

\begin{theorem}\label{thm:grad_hess_lip}
Let $\bm{s}_j$ be the global minimizer of the TR subproblem at the $j$-th iteration of Algorithm~\ref{alg:unified_imex_tr}, and let sequence $\{\bm{v}_j\}$ be generated by Algorithm~\ref{alg:unified_imex_tr}. Then the gradient sequence $\{\nabla \bar{E}(\bm{v}_j)\}$ is bounded and the function $\nabla \bar{E}$ is Lipschitz continuous with a scalar Lipschitz constant $g_{\text{Lip}} > 0$ in an open convex set containing the iterate sequence $\{\bm{v}_j\}$ and $\{\bm{v}_j+\bm{s}_j\}$. Furthermore, the Hessian function $\nabla^2 \bar{E}$ is Lipschitz continuous on a path defined by the sequence of $\{\bm{v}_j\}$ and $\{\bm{v}_j+\bm{s}_j\}$.
\end{theorem}
\begin{proof}
By the continuity and coercivity of $\bar{E}$ and the decrease of $\bar{E}(\bm{v}_j)$, the sublevel set $\mathcal{L} = \{\bm{v} \mid \bar{E}(\bm{v}) \leq \bar{E}(\bm{v}_0)\}$ is compact and $\{\bm{v}_j\}\subset \mathcal{L}$, hence the sequences $\{\bm{v}_j\}$ and $\{\nabla \bar{E}(\bm{v}_j)\}$ are bounded. As the trial step $\bm{s}_j$ is bounded $\|\bm{s}_j\|_2 \leq r_j$, thus $\{\bm{v}_j + \bm{s}_j\}$ is also bounded. By the boundedness of $\{\bm{v}_j\}$ and $\{\bm{v}_j + \bm{s}_j\}$, there exists $R>0$ such that $\|\bm{v}_j\|_2 \leq R$ and $\|\bm{v}_j + \bm{s}_j\|_2 \leq R$ for all $j$. Let $\mathcal{B}:=\{\bm{v}\mid \|\bm{v}\|_2<R+\delta\}$, where $\delta > 0$, then $\mathcal{B}$ is a bounded open convex set containing $\{\bm{v}_j\}$ and $\{\bm{v}_j + \bm{s}_j\}$. 

By the mean value inequality, for any $\bm{u}, \bm{v} \in \mathcal{B}$,
\[\|\nabla \bar{E}(\bm{u}) - \nabla \bar{E}(\bm{v})\| \leq \sup_{\bm{w}\in \bar{\mathcal{B}}} \|\nabla^2 \bar{E}(\bm{w})\|\|\bm{u} - \bm{v}\|_2,\]
where $\bar{\mathcal{B}}$ is the closure of $\mathcal{B}$. By the continuity of $\nabla^2 \bar{E}$ and the compactness of $\bar{\mathcal{B}}$, we have $\sup_{\bm{w}\in \bar{\mathcal{B}}} \|\nabla^2 \bar{E}(\bm{w})\| < \infty$, with $g_{\text{Lip}} := \sup_{\bm{w}\in \bar{\mathcal{B}}} \|\nabla^2 \bar{E}(\bm{w})\|$, the Lipschitz continuity of $\nabla \bar{E}$ on $\mathcal{B}$ follows.

To prove the Lipschitz continuity of $\nabla^2 \bar{E}$ on the path defined by the sequences $\{\bm{v}_j\}$ and $\{\bm{v}_j + \bm{s}_j\}$, we only need to show that $\nabla^2 \bar{E}$ is Lipschitz continuous on $\bar{\mathcal{B}}$. According to the mean value inequality, for any $\bm{u}, \bm{v} \in \bar{\mathcal{B}}$, we have
\[
\|\nabla^2\bar E(\bm{u})-\nabla^2\bar E(\bm{v})\|
\leq \sup_{\bm{w}\in\bar{\mathcal{B}}} \|\nabla^3\bar E(\bm{w})\|\|\bm{u}-\bm{v}\|_2.
\]
As the third-order derivative $\nabla^3 \bar{E}$ exists and is continuous, it follows that
\[
\|\nabla^2\bar E(\bm{u})-\nabla^2\bar E(\bm{v})\| \leq h_{\text{Lip}}\|\bm{u}-\bm{v}\|_2,
\]
where $h_{\text{Lip}}:=\sup_{\bm{w}\in\bar{\mathcal{B}}} \|\nabla^3\bar E(\bm{w})\| < \infty$. This completes the proof.
\end{proof}

Combining the model-specific Lipschitz estimates in Theorem~\ref{thm:grad_hess_lip}, the global optimality of the TR  subproblem solver in Theorem~\ref{thm:global_optimality}, and the TRACE convergence results in~\cite{curtis2017trust}, we obtain the following outer convergence result.
\begin{theorem}[Theorem 3.14, Theorem 3.26 \cite{curtis2017trust}]
Let $\{ \bm{v}_j \}$ be the sequence generated by Algorithm~\ref{alg:unified_imex_tr}. Suppose that the TR  subproblem at each iteration is solved to global optimality as guaranteed by Theorem~\ref{thm:global_optimality} under the conditions specified therein. Then
 $$\lim_{j \to \infty} \|\nabla \bar{E}(\bm{v}_j)\| = 0, \quad \lim\inf_{j\rightarrow\infty}\nabla^2 \bar{E}(\bm{v}_j) \succeq 0.$$
\end{theorem}

\subsection{Efficient Global Solver for the Nonconvex TR Subproblem}\label{subsec:subproblem_solver}
In this subsection, we propose an efficient IMEX algorithm to solve the nonconvex TR subproblem \eqref{eq:tr_subproblem_with_subscript} and analyze the algorithm's ability to converge to the global minimizer of the subproblem.

At each iteration of the IMEX-TR, a core challenge lies in solving the nonconvex subproblem \eqref{eq:tr_subproblem_with_subscript} globally and efficiently. For ease of analysis, we remove the subscript $j$ and the constant term, and restate the subproblem \eqref{eq:tr_subproblem_with_subscript} as
\begin{equation}\label{eq:tr_subproblem}
\min_{\|\bm{d}\|_2\leq r} f(\bm{d}) = \bm{g}^{\top}\bm{d} + \frac{1}{2} \bm{d}^{\top}\bm{H}\bm{d}.
\end{equation}

We exploit the structural properties of the Hessian for the LB model, which can be decomposed as
\[ \bm{H} = \underbrace{\bm{P} \nabla^2 G(\bm{P}^{\top} \bm{v}) \bm{P}^{\top}}_{\bm{D}} + \underbrace{\bm{P} \nabla^2 F(\bm{P}^{\top} \bm{v}) \bm{P}^{\top}}_{\bm{T}} = \bm{D} + \bm{T}.\]
Here, $\bm{D}$ is a positive semi-definite diagonal matrix in reciprocal space.

We develop an algorithm that embeds an adaptive implicit-explicit iteration within the TR loop. The implicit-explicit scheme for the subproblem is
\begin{equation}\label{eq:imex-scheme}
(\bm{I} + \eta (\bm{D} + \lambda_{k+1} \bm{I})) \bm{d}_{k+1} = \bm{d}_k - \eta (\bm{g} + \bm{T} \bm{d}_k),\quad k=0,1,2,\dots
\end{equation}
where $\eta > 0$ is a step size, and $\lambda_{k+1}\geq 0$ is a Lagrange multiplier adjusted to ensure $\|\bm{d}_{k+1}\| \leq r_j$. In this scheme, the diagonal part $\bm{D}$ is treated implicitly while the dense part  $\bm{T}$ is treated explicitly, enabling efficient FFT-based computation while maintaining stability.

By letting $\bm{L}_{k+1}=\bm{D}+\lambda_{k+1}\bm{I}$, $\bar{\bm{g}}_k = \bm{g} + \bm{T} \bm{d}_k$, and substituting these into \eqref{eq:imex-scheme}, then it can be rewritten as
\begin{equation}\label{eq:semi_implicit_2linear_system}
\underbrace{(\bm{I} + \eta \bm{L}_{k+1})}_{\text{Diagonal matrix}} \bm{d}_{k+1} = \bm{d}_k - \eta \bar{\bm{g}}_k.
\end{equation}

The features of scheme \eqref{eq:semi_implicit_2linear_system} are
\begin{itemize}
    \item The matrix $\bm{I} + \eta \bm{L}_{k+1}$ is a diagonal matrix in reciprocal space, thus the update process only involves element-wise operations on the matrix, resulting in low computational complexity.
    \item The coefficient matrix $\bm{I} + \eta \bm{L}_{k+1}$ is always positive definite, which guarantees the existence and uniqueness of the solution to the linear system, ensuring the numerical stability of the scheme.
    \item By adjusting the value of $\lambda_{k+1}$, the norm of the solution can be limited within a certain range to ensure that the solution is within the feasible domain of the subproblem.
\end{itemize}

\begin{remark}
In order to exploit the diagonal structure of $\bm{T}$ in physical space, the product $\bm{T}\bm{d}$ is efficiently computed by performing point-wise multiplication in physical space followed by the FFT. This approach ensures an overall computational complexity of $\mathcal{O}(N \log N)$ per iteration, avoiding the $\mathcal{O}(N^2)$ cost associated with direct matrix-vector multiplication in reciprocal space.
\end{remark}

We first show that the solution of \eqref{eq:semi_implicit_2linear_system} satisfies $\|\bm{d}\|_2\leq r$ by adaptively adjusting the value of $\lambda_{k+1}$, which ensures the solution of the adaptive implicit-explicit iteration is a feasible solution to problem \eqref{eq:tr_subproblem}.

Since $\bm{I} + \eta \bm{L}_{k+1}$ is a diagonal matrix and positive definite, we get
\[
\bm{d}_{k+1} = (\bm{I} + \eta \bm{L}_{k+1})^{-1}\bm{b}_k =
\begin{pmatrix}
[\bm{b}_k]_1/(1 + \eta(\bm{D}_{1,1}+\lambda_{k+1}))\\
\vdots \\
[\bm{b}_k]_{N-1}/(1 + \eta(\bm{D}_{N-1,N-1}+\lambda_{k+1}))
\end{pmatrix},
\]
where $[\bm{b}_k]_i$ denotes the $i$-th component of $\bm{b}_k=\bm{d}_k-\eta \bar{\bm{g}}_k$ and $D_{i,i}$ is the $i$-th diagonal element of the diagonal matrix $\bm{D}$, thus
\begin{equation*}
\phi_k(\lambda_{k+1}) = \|\bm{d}_{k+1}\|_2^2 = 
\sum_{i=1}^{N-1}\left(\frac{[\bm{b}_k]_i}{(1 + \eta(\bm{D}_{i,i}+\lambda_{k+1}))}\right)^2,
\end{equation*}
it is easy to see that $\phi_k(\lambda)$ is a continuously differentiable and monotonically decreasing function with respect to $\lambda$, then there must exist $0 \leq \lambda_{\min} \leq \lambda_{\max} < \infty$, such that for any $\lambda_{k+1}\in[\lambda_{\min}, \lambda_{\max}]$, we have $\phi_k(\lambda_{k+1})\leq r^2$. We set the value of $\lambda_{k+1}$ as
\begin{equation}\label{eq:lambda_update_rule}
    \lambda_{k+1} = \begin{cases}
    0,\ &\text{if}\ ~\phi_k(0) \leq r^2,\\
    \lambda^{'},\ &\text{if}\ ~\phi_k(0) > r^2,
    \end{cases}
\end{equation}
where $\lambda^{'}$ satisfying $\phi_k(\lambda^{'}) = r^2$ is computed via Newton's method
\[\lambda_{\iota+1} = \lambda_\iota - (\phi_k(\lambda_\iota) - r^2)/\phi_k^{\prime}(\lambda_\iota), \iota=0,1,2,\dots,\]
with $\lambda_0 = 0$. The complete adaptive implicit-explicit iteration for solving the TR subproblem \eqref{eq:tr_subproblem} is summarized in Algorithm~\ref{alg:adaptive-implicit-explicit}.

\begin{algorithm}[htbp]
\renewcommand{\algorithmicrequire}{\textbf{Input:}}
\renewcommand{\algorithmicensure}{\textbf{Output:}}
\caption{Adaptive implicit-explicit iteration}
\label{alg:adaptive-implicit-explicit}
\begin{algorithmic}
\Require: Step size $\eta>0$, trust region radius $r>0$, gradient $\bm{g}$, decompose Hessian as $\bm{H}=\bm{D}+\bm{T}$, initial value $\bm{d}_0=-r\bm{g}/\|\bm{g}\|_2$, tolerance $\varepsilon_{\text{sub}}$.
\Ensure $\bm{d}_{k+1}$.
\For {$k = 0,1,\dots$}
    \State $\lambda_{k+1} = 0$.
    \While {$\phi_k(\lambda_{k+1}) > r^2$}
        \Statex \hspace*{\algorithmicindent} Find $\lambda_{k+1}$ satisfies $\phi_k(\lambda_{k+1})=r^2$ by Newton's method
        \State $\lambda_{k+1} \leftarrow \lambda_{k+1} - (\phi_k(\lambda_{k+1}) - r^2)/ \phi^{\prime}_k(\lambda_{k+1})$
    \EndWhile
    \State $\bm{d}_{k+1} = (\bm{I}+\eta(\bm{D}+\lambda_{k+1}\bm{I}))^{-1}(\bm{d}_k - \eta(\bm{g} + \bm{T}\bm{d}_k))$
    \If{$\|\bm{g} + (\bm{H} + \lambda_{k+1}\bm{I})\bm{d}_{k+1} \|_2 < \varepsilon_{\text{sub}}$}
        \State \Return $\bm{d}_{k+1}$.
    \EndIf
\EndFor
\end{algorithmic}
\end{algorithm}

By Algorithm~\ref{alg:adaptive-implicit-explicit} and the update rule \eqref{eq:lambda_update_rule}, we have the following corollary.
\begin{corollary}\label{coro:feasible_kkt}
The sequence $\{(\lambda_k, \bm{d}_k)\}$ generated by Algorithm~\ref{alg:adaptive-implicit-explicit} has the following properties,
    \[\lambda_{k}\geq 0,\quad \|\bm{d}_{k}\|_2\leq r,\quad \lambda_{k}(\|\bm{d}_{k}\|_2 - r) = 0.\]
\end{corollary}

By the above Corollary \ref{coro:feasible_kkt}, we need only focus on the remaining two conditions in the
optimality conditions \eqref{eq:tr_subproblem_optimal_condition}. In the following, we first focus on the condition $\bm{g} + (\bm{H} + \lambda^*\bm{I})\bm{d}^* = 0$, and then the condition $\bm{H} + \lambda^* \bm{I} \succeq 0$.

\subsubsection{Global Convergence of the IMEX solver for the TR subproblem}
We analyze the properties of the sequence $\{\bm{d}_k\}$ generated by Algorithm~\ref{alg:adaptive-implicit-explicit} for solving the nonconvex subproblem. 

Let $\beta^2 = \lambda \geq 0$ and define the Lagrangian function associated with the TR subproblem \eqref{eq:tr_subproblem} as
\begin{equation}\label{eq:lagrangian_function}
L(\bm{d}, \beta) = f(\bm{d}) + \frac{\beta^2}{2}(\|\bm{d}\|_2^2-r^2).
\end{equation}
We first present the following Lemma~\ref{lemma:function_value_decrease_tr} to describe the decreasing property of function values of $L$, which plays a crucial role in proving the convergence for the scheme \eqref{eq:semi_implicit_2linear_system}.

\begin{lemma}\label{lemma:function_value_decrease_tr}
For any $a > 0$, if step size $\eta$ satisfies the inequality $\eta^{-1} \geq \|\bm{T}\| + 2a$,
then we get the descent property for the function values of $f$ in \eqref{eq:tr_subproblem} and $L$ in \eqref{eq:lagrangian_function},
\[
\begin{aligned}
&f(\bm{d}_k) - f(\bm{d}_{k+1})\geq a\|\bm{d}_{k+1}-\bm{d}_k\|_2^2,\\
&L(\bm{d}_k, \beta_k) - L(\bm{d}_{k+1}, \beta_{k+1})\geq a\|\bm{d}_{k+1}-\bm{d}_k\|_2^2.
\end{aligned}
\]
\begin{proof}
We define an auxiliary function and show that the update rule \eqref{eq:semi_implicit_2linear_system} is equivalent to minimizing this auxiliary function, \textit{i.e.},
\begin{equation}\label{eq:auxiliary_function}
\bm{d}_{k+1} = \argmin_{\bm{d}} \{ f_k(\bm{d}) \},
\end{equation}
where the objective function is defined as
\[
f_k(\bm{d}) = \frac{1}{2} \bm{d}^{\top}\bm{L}_{k+1} \bm{d} + (\bm{g}+\frac{1}{2}\bm{T}\bm{d}_k)^{\top}\bm{d}_k + \bar{\bm{g}}_k^{\top}(\bm{d} - \bm{d}_k) + \frac{1}{2 \eta} \lVert \bm{d} - \bm{d}_k \rVert_2^2,
\]
By the optimality of $\bm{d}_{k+1}$ in \eqref{eq:auxiliary_function}, we have
\begin{align*}
&f(\bm{d}_{k}) + \frac{\beta_{k+1}^2}{2}\|\bm{d}_{k}\|_2^2 = f_k(\bm{d}_{k}) \geq f_k(\bm{d}_{k+1})\\
=& \frac12 (\bm{d}_{k+1})^{\top}\bm{L}_{k+1}(\bm{d}_{k+1}) + (\bm{g}+\frac12\bm{T}\bm{d}_{k})^{\top}(\bm{d}_{k}) + (\bar{\bm{g}}_k)^{\top}(\bm{d}_{k+1}-\bm{d}_{k}) + \frac{1}{2\eta}\|\bm{d}_{k+1}-\bm{d}_{k}\|_2^2\\
=& \frac{1}{2} (\bm{d}_{k+1})^{\top}(\bm{D}\bm{d}_{k+1}) + (\bm{g})^{\top}(\bm{d}_{k+1}) + (\bm{T}\bm{d}_{k+1})^{\top}(\bm{d}_{k+1}-\bm{d}_{k}) + \frac{1}{2\eta}\|\bm{d}_{k+1}-\bm{d}_{k}\|_2^2\\
& - (\bm{T}(\bm{d}_{k+1}-\bm{d}_{k}))^{\top}(\bm{d}_{k+1}-\bm{d}_{k}) + \frac{1}{2}(\bm{T}\bm{d}_{k})^{\top}(\bm{d}_{k}) + \frac{\beta_{k+1}^2}{2} \|\bm{d}_{k+1}\|_2^2\\
=& f(\bm{d}_{k+1}) + \frac{\beta_{k+1}^2}{2} \|\bm{d}_{k+1}\|_2^2 + \frac{1}{2\eta}\|\bm{d}_{k+1}-\bm{d}_{k}\|_2^2 + \frac12(\bm{T}\bm{d}_{k+1})^{\top}(\bm{d}_{k+1}) \\
& - (\bm{T}(\bm{d}_{k+1}-\bm{d}_{k}))^{\top}(\bm{d}_{k+1}-\bm{d}_{k}) + \frac12(\bm{T}\bm{d}_{k})^{\top}(\bm{d}_{k}) - (\bm{T}\bm{d}_{k+1})^{\top}(\bm{d}_{k})\\
=& f(\bm{d}_{k+1}) + \frac{\beta_{k+1}^2}{2} \|\bm{d}_{k+1}\|_2^2 + \frac{1}{2\eta}\|\bm{d}_{k+1}-\bm{d}_{k}\|_2^2 - \frac12 (\bm{T}(\bm{d}_{k+1}-\bm{d}_{k}))^{\top}(\bm{d}_{k+1}-\bm{d}_{k})\\
=& f(\bm{d}_{k+1}) + \frac{\beta_{k+1}^2}{2} \|\bm{d}_{k+1}\|_2^2 + \frac{1}{2}(\bm{d}_{k+1}-\bm{d}_{k})^{\top}(\frac{1}{\eta}I-\bm{T})(\bm{d}_{k+1}-\bm{d}_{k}).\\
\end{align*}
Since $\|\bm{d}_{k}\|_2^2\leq r^2$, it follows that $f(\bm{d}_{k})+\frac{\beta_{k+1}^2}{2}\|\bm{d}_{k}\|_2^2 \leq f(\bm{d}_{k})+\frac{\beta_{k+1}^2}{2}r^2$.

By combining the above inequality with Corollary~\ref{coro:feasible_kkt}, we have
\[
    f(\bm{d}_{k})+\frac{\beta_{k+1}^2}{2}r^2 \geq f(\bm{d}_{k+1}) + \frac{\beta_{k+1}^2}{2}r^2 + \frac{1}{2}\braket{\bm{d}_{k+1}-\bm{d}_{k},(\frac{1}{\eta}I-\bm{T})(\bm{d}_{k+1}-\bm{d}_{k})}.
\]
Thus, when $1/\eta \geq \|\bm{T}\| + 2a$, it reads $f(\bm{d}_{k}) - f(\bm{d}_{k+1}) \geq a\|\bm{d}_{k+1}-\bm{d}_{k}\|_2^2$,
and
\[
\begin{aligned}
&L(\bm{d}_{k}, \beta_k) - L(\bm{d}_{k+1}, \beta_{k+1})\\
=&f(\bm{d}_{k})+\frac{\beta_k^2}{2}(\|\bm{d}_{k}\|_2^2-r^2) - \left(f(\bm{d}_{k+1}) +\frac{\beta_{k+1}^2}{2}(\|\bm{d}_{k+1}\|_2^2-r^2)\right)\\
=& f(\bm{d}_{k}) - f(\bm{d}_{k+1}) \geq  a\|\bm{d}_{k+1}-\bm{d}_{k}\|_2^2.
\end{aligned}
\]

The second equality above is derived from Corollary~\ref{coro:feasible_kkt}.
\end{proof}
\end{lemma}

Next, we show that the norm of the gradient is bounded, which is also important to our convergence theorem.
\begin{lemma}\label{lemma:gradient_bound_tr}
If the step size $\eta$ satisfies
\begin{equation}\label{eq:step_cond2}
0< \eta_{\min} \leq \eta \leq \frac{1}{\|\bm{T}\|+2a},    
\end{equation}
then there must exist $b>0$, such that
\begin{equation}\label{eq:gradient_bound_tr}
\left\|\begin{pmatrix}
\nabla_d L(\bm{d}_{k+1}, \beta_{k+1})\\
\nabla_{\beta} L(\bm{d}_{k+1}, \beta_{k+1})
\end{pmatrix}\right\|_2 \leq 
b\left\|\bm{d}_{k+1}-\bm{d}_{k}\right\|_2.
\end{equation}
\end{lemma}

\begin{proof}
\[\begin{aligned}
&\left\|\begin{pmatrix}
\nabla_{\bm{d}} L(\bm{d}_{k+1}, \beta_{k+1})\\
\nabla_{\beta} L(\bm{d}_{k+1}, \beta_{k+1})
\end{pmatrix}\right\|_2 = \left\|\begin{pmatrix}
\nabla f(\bm{d}_{k+1}) + \beta_{k+1}^2\bm{d}_{k+1}\\
\beta_{k+1}(\|\bm{d}_{k+1}\|_2^2-r^2)
\end{pmatrix}\right\|_2 \\
=& \left\|\begin{pmatrix}
\nabla f(\bm{d}_{k+1}) + \beta_{k+1}^2\bm{d}_{k+1}\\
0
\end{pmatrix}\right\|_2 \\
=&\|\bm{g}+\bm{T}\bm{d}_{k}+\bm{D}\bm{d}_{k+1}+\beta_{k+1}^2\bm{d}_{k+1} + \bm{T}(\bm{d}_{k+1}-\bm{d}_{k})\|_2\\
=&\|-(\bm{d}_{k+1}-\bm{d}_{k})/\eta + \bm{T}(\bm{d}_{k+1}-\bm{d}_{k})\|_2\\
\leq& ({1/\eta+\|\bm{T}\|})\|\bm{d}_{k+1}-\bm{d}_{k}\|_2\\
\leq& ({1/\eta_{\min}+\|\bm{T}\|})\|\bm{d}_{k+1}-\bm{d}_{k}\|_2.
\end{aligned}\]

The second equality above follows from Corollary~\ref{coro:feasible_kkt} and $\beta_{k+1}^2 = \lambda_{k+1}$, \eqref{eq:imex-scheme} has been used for the fourth equality. Let $b=1/\eta_{\min}+\|\bm{T}\|$ finishes the proof.
\end{proof}

Using the K\L~property and the techniques of Lemma 2.6 from \cite{attouch2013convergence}, we prove that the sequence $\{\bm{d}_{k}\}$ is a Cauchy sequence.

\begin{definition}[K\L~Property\cite{bolte2014proximal}]
Let $L:\mathbb{C}^N\rightarrow  (-\infty,+\infty]$ be a proper and lower semi-continuous function. For each \(\bm{z}_* \in \text{dom } \nabla L\), there exist \(\kappa > 0\), a neighborhood U of \(\bm{z}_*\), and a concave continuous function $\varphi:[0,\kappa)\rightarrow \mathbb{R}_+$, such that 
\begin{equation*}
  \varphi(0) = 0,\ \varphi\in \mathbb{C}^1(0,\kappa),\ \varphi^{\prime}(s)> 0\ \text{for\ all}\ s\in (0,\kappa),\ \text{and}
\end{equation*}
\begin{equation*}
\varphi'(L(\bm{z}) - L(\bm{z}_*)) \text{dist}(0, \nabla L(\bm{z})) \geq 1,
\end{equation*}
for all $\bm{z} \in U \cap \{ \bm{z} : L(\bm{z}_*) < L(\bm{z}) < L(\bm{z}_*) + \kappa \}$, then $L(\bm{z})$ is said to satisfy the K\L~property at point \(\bm{z}_*\). If \(L(\bm{z})\) satisfies the K\L~property at every point $\bm{z}$, then \(L(\bm{z})\) is called a K\L~function.
\end{definition}

\begin{theorem}\label{thm:cauchy_sequence}
Let the sequence $\{\bm{z}_k\}=\{(\bm{d}_{k},\beta_k)\}=\{(\bm{d}_{k},\sqrt{\lambda}_k)\}$ be generated by the IMEX subproblem solver in Algorithm~\ref{alg:unified_imex_tr}. If the step condition \eqref{eq:step_cond2} is satisfied, then we have that
\[
\sum_{k=1}^{\infty}\left\|\bm{d}_{k}-\bm{d}_{k+1}\right\|_2<\infty,
\]
which means $\{\bm{d}_{k}\}$ is a Cauchy sequence and that $\{\bm{d}_{k}\}$ converges to $ \bm{d}_{*}$ as $k\to\infty$.
\end{theorem}\vspace*{-0.05in}
\begin{proof}
As $\{\bm{z}_{k}\}$ is bounded, there must exist an accumulation point of $\{\bm{z}_k\}$, let it be $ \bm{z}_{*}=( \bm{d}_*,  \beta_*)$, and there exists a subsequence $\{\bm{z}_{k_l}\}$ of $\{\bm{z}_{k}\}$ converging to $ \bm{z}_*$ as $l\to\infty$. Due to the boundedness from below of $L$ and the monotonicity of $\{L(\bm{z}_k)\}$, the continuous property of $L$ implies that $\{L(\bm{z}_k)\}$ converges to $L( \bm{z}_*)$ as $k\to\infty$ with $L( \bm{z}_*) < L(\bm{z}_k)$, $k\in \mathbb{N}$. In particular, $L$ satisfies the K\L~property~\cite{bolte2014proximal}, and thus
\[
L( \bm{z}_*) < L(\bm{z}_{k}) < L( \bm{z}_*) +\kappa\;\mbox{ for large }\;k\in\mathbb N\;\mbox{ and small }\kappa >0.
\]
Taking the latter into account, we define a sequence $\{b_k \}$ by
\[
b_k:= \left\| \bm{z}_k -  \bm{z}_* \right\|_2 +2\big(a^{-1} (L(\bm{z}_{k}) - L( \bm{z}_{*}))  \big)^{\frac{1}{2}} +b a^{-1} \varphi\big(L(\bm{z}_{k}) - L( \bm{z}_{*})\big)
\]
and deduce from the continuity of $\varphi$ that the origin $0\in\mathbb{R}$ is an accumulation point of $\{b_k\}$. Thus there exists $k_0 := k_{l_0}$ such that the inequalities
\begin{equation}\label{converg 4.2}
\left\| \bm{z}_{k_0} -  \bm{z}_{*} \right\|_2 + 2\big(a^{-1}(L(\bm{z}_{k_0}) - L( \bm{z}_{*}))\big)^{\frac{1}{2}} +b a^{-1}\varphi\big(L(\bm{z}_{k_0}) - L( \bm{z}_{*})\big) < \epsilon,
\end{equation}
\[
L( \bm{z}_{*}) < L(\bm{z}_{k_0}) < L( \bm{z}_{*}) + \kappa
\]
are satisfied and yield in turn the inclusions
\[
\bm{z}_{k_0} \in B( \bm{z}_{*},\epsilon) \cap \big[L( \bm{z}_{*}) < L < L( \bm{z}_{*})+\kappa\big] \subset U \cap \big[L( \bm{z}_{*}) < L < L( \bm{z}_{*}) + \kappa\big],
\]
hence $\bm{d}_{k_0} \in B( \bm{d}_{*}, \epsilon)$.

Moreover, we claim that if $L(\bm{z}_{k})<L( \bm{z}_{*}) + \kappa$ and $\bm{d}_{k}\in B( \bm{d}_{*}, \epsilon)$, then
\begin{equation}\label{eq:success_delta_sequence_bound}
2\|\bm{d}_{k+1}-\bm{d}_{k}\|_2 \leq \|\bm{d}_{k} - \bm{d}_{k-1}\|_2 + \frac{b}{a}[\varphi\big(L(\bm{z}_{k}) - L( \bm{z}_{*})\big) - \varphi\big( L(\bm{z}_{k+1}) - L( \bm{z}_{*})\big)].
\end{equation}
By Lemma~\ref{lemma:function_value_decrease_tr}, we have the inequality
\begin{equation}\label{eq:decrease}
L(\bm{z}_{k}) - L(\bm{z}_{k+1})\geq a\|\bm{d}_{k+1}-\bm{d}_{k}\|_2^2,
\end{equation}
hence $0\leq L(\bm{z}_{k_0 +1}) - L( \bm{z}_{*}) \leq L(\bm{z}_{k_0}) - L( \bm{z}_{*}) < \kappa$,
which implies therefore the inequality
\begin{equation*}
\varphi\big(L(\bm{z}_{k_0}) - L( \bm{z}_{*})\big) - \varphi\big( L(\bm{z}_{k_0+1}) - L( \bm{z}_{*})\big)\geq \varphi'\big(L(\bm{z}_{k_0}) - L( \bm{z}_{*})\big)\big( L(\bm{z}_{k_0}) -L(\bm{z}_{k_0+1})\big)
\end{equation*}
by the concavity of $\varphi$. Combining \eqref{eq:decrease} and 
\[
\varphi'(L(\bm{z}_{k_0}) - L( \bm{z}_{*}))\geq \frac{1}{\nabla L(\bm{d}_{k_0}, \beta_{k_0})}\geq \frac{1}{b\|\bm{d}_{k_0} - \bm{d}_{k_0-1}\|_2},
\]
we have
\[
\frac{b}{a}[\varphi\big(L(\bm{z}_{k_0}) - L( \bm{z}_{*})\big) - \varphi\big( L(\bm{z}_{k_0+1}) - L( \bm{z}_{*})\big)]\geq\frac{\|\bm{d}_{k_0+1}-\bm{d}_{k_0}\|_2^2}{\|\bm{d}_{k_0} - \bm{d}_{k_0-1}\|_2}.
\]
By $2\sqrt{\alpha\beta}\leq \alpha + \beta$, we get \eqref{eq:success_delta_sequence_bound} with $k=k_0$, \textit{i.e.},
\[
2\|\bm{d}_{k_0+1}-\bm{d}_{k_0}\|_2 \leq \|\bm{d}_{k_0} - \bm{d}_{k_0-1}\|_2 + \frac{b}{a}[\varphi\big(L(\bm{z}_{k_0}) - L( \bm{z}_{*})\big) - \varphi\big( L(\bm{z}_{k_0+1}) - L(\bm{z}_{*})\big)].
\]
As we have 
\begin{equation}\label{eq:inequality_induced_by_fx_decrease}
\|\bm{d}_{k_0+1} - \bm{d}_{k_0} \|_2 \leq \sqrt{\frac{L(\bm{z}_{k_0}) - L(\bm{z}_{k_0+1})}{a}}\leq \sqrt{\frac{L(\bm{z}_{k_0}) - L( \bm{z}_{*})}{a}}
\end{equation}
by \eqref{eq:decrease}, hence 
\[
\|\bm{d}_{k_0+1}- \bm{d}_{*}\|_2 \leq \|\bm{d}_{k_0}  - \bm{d}_{*}\|_2 + \|\bm{d}_{k_0+1}-\bm{d}_{k_0}  \|_2 \leq \|\bm{d}_{k_0}  - \bm{d}_{*}\|_2+\sqrt{\frac{L(\bm{z}_{k_0})-L( \bm{z}_{*})}{a}}<\epsilon,
\]
then $\bm{d}_{k_0+1} \in B( \bm{d}_{*},\epsilon)$, we get \eqref{eq:success_delta_sequence_bound} with $k=k_0+1$, \textit{i.e.},
\[
2\|\bm{d}_{k_0+2}-\bm{d}_{k_0+1}\|_2 \leq \|\bm{d}_{k_0+1} - \bm{d}_{k_0} \|_2 + \frac{b}{a}[\varphi\big(L(\bm{z}_{k_0+1}) - L( \bm{z}_{*})\big) - \varphi\big( L(\bm{z}_{k_0+2}) - L( \bm{z}_{*})\big)].
\]

Next, for $l = 1,2,\cdots$, we prove $\bm{d}_{k_0+l}\in B( \bm{d}_{*}, \epsilon)$ and 
\begin{equation}\label{eq:summary_finite}
\begin{aligned}
&\sum_{i=1}^{l}\|\bm{d}_{k_0+i+1} - \bm{d}_{k_0+i}\|_2 + \|\bm{d}_{k_0+l+1} - \bm{d}_{k_0+l}\|_2\\
\leq& \|\bm{d}_{k_0+1}-\bm{d}_{k_0} \|_2 + \frac{b}{a}[\varphi\big(L(\bm{z}_{k_0+1}) - L( \bm{z}_{*})\big) - \varphi\big( L(\bm{z}_{k_0+l+1}) - L( \bm{z}_{*})\big)]
\end{aligned}
\end{equation}
by induction on $l$. Suppose $\bm{d}_{k_0+l}\in B( \bm{d}_{*}, \epsilon)$ and \eqref{eq:summary_finite} hold for some $l\geq 1$, using the triangle inequality and \eqref{eq:summary_finite} we have
\[\begin{aligned}
&\| \bm{d}_{*} - \bm{d}_{k_0+l+1}\|_2 \leq \| \bm{d}_{*} - \bm{d}_{k_0} \|_2 + \|\bm{d}_{k_0}  - \bm{d}_{k_0+1}\|_2 + \sum_{i=1}^{l}\|\bm{d}_{k_0+i+1} - \bm{d}_{k_0+i}\|_2\\
\leq&\| \bm{d}_{*} - \bm{d}_{k_0} \|_2 + 2\|\bm{d}_{k_0}  -\bm{d}_{k_0+1}\|_2\\ 
& +\frac{b}{a}[\varphi\big(L(\bm{z}_{k_0+1}) - L( \bm{z}_{*})\big) - \varphi\big( L(\bm{z}_{k_0+l+1}) - L( \bm{z}_{*})\big)]\\
\leq&\| \bm{d}_{*} - \bm{d}_{k_0} \|_2 + 2\|\bm{d}_{k_0}  -\bm{d}_{k_0+1}\|_2+\frac{b}{a}[\varphi\big(L(\bm{z}_{k_0}) - L( \bm{z}_{*})\big)].
\end{aligned}\]
Using the above inequality, \eqref{eq:inequality_induced_by_fx_decrease} and \eqref{converg 4.2} we conclude that $\bm{d}_{k_0+l+1}\in B( \bm{d}_{*}, \epsilon)$. Hence, \eqref{eq:success_delta_sequence_bound} holds with $k=k_0+l+1$, \textit{i.e.},
\[\begin{aligned}
2\|\bm{d}_{k_0+l+2}-\bm{d}_{k_0+l+1}\|_2 \leq &\|\bm{d}_{k_0+l+1} - \bm{d}_{k_0+l}\|_2 + \\
&\frac{b}{a}[\varphi\big(L(\bm{z}_{k_0+l+1}) - L( \bm{z}_{*})\big) - \varphi\big( L(\bm{z}_{k_0+l+2}) - L( \bm{z}_{*})\big)].
\end{aligned}\]
Adding the above inequality with \eqref{eq:summary_finite}, yields \eqref{eq:summary_finite} with $l+1$, which completes the induction proof.

Direct use of \eqref{eq:summary_finite} shows that
\[
\sum_{i=k_0+1}^{k_0+l}\|\bm{d}_{i+1} - \bm{d}_{i}\|_2
\leq \|\bm{d}_{k_0+1}-\bm{d}_{k_0}  \|_2 + \frac{b}{a}(\varphi\big(L(\bm{z}_{k_0+1}) - L( \bm{z}_{*})\big)),
\]
therefore $\sum_{i=1}^{\infty}\|\bm{d}_{i+1}-\bm{d}_i\|_2 < +\infty$, 
which implies that the sequence $\{\bm{d}_{k}\}$ is a Cauchy sequence, which therefore converges to $ \bm{d}_{*}$ since it is an accumulation point of $\{\bm{d}_{k}\}$.
\end{proof}

Thus, we can prove that any limit point $\{(\bm{d}_{*},\lambda_*)\}$ of sequence $\{(\bm{d}_{k},\lambda_k)\}$ is a KKT point.
\begin{theorem}\label{thm:converge_to_kkt}
When step size condition \eqref{eq:step_cond2} is satisfied, every limit point $\{(\bm{d}_{*},\lambda_*)\}$ of the sequence $\{(\bm{d}_{k},\lambda_k)\}$ is a KKT point of problem \eqref{eq:tr_subproblem}.
\end{theorem}
\begin{proof}
As $\{\lambda_k\}$ is a bounded sequence, by the Bolzano-Weierstrass theorem, there exists a subsequence $\{\lambda_{k_l}\}$ such that $\lambda_{k_l}\rightarrow \lambda_*$ as $l\rightarrow \infty$. By Theorem~\ref{thm:cauchy_sequence} we know that ${\bm{d}_{k}}$ converges to $ \bm{d}_{*}$, thus $\lim\limits_{k\rightarrow\infty}\|\bm{d}_{k+1}-\bm{d}_{k}\|_2 = 0$. Taking the limits along the convergent subsequence on both sides of \eqref{eq:gradient_bound_tr} yields
\[
\lim_{l\rightarrow\infty}\|\bm{g} + \bm{D} \bm{d}_{k_l+1} + \bm{T} \bm{d}_{k_l} + \lambda_{k_l+1} \bm{d}_{k_l+1}\|_2 = 0\Longrightarrow \bm{g} + \bm{D} \bm{d}_{*} + \bm{T} \bm{d}_{*} + \lambda_* \bm{d}_* = 0,
\]
combining with Corollary~\ref{coro:feasible_kkt}, we conclude that $(\bm{d}_{*}, \lambda_*)$ is a KKT point of problem \eqref{eq:tr_subproblem}.
\end{proof} 

Finally, we establish the global convergence of Algorithm~\ref{alg:adaptive-implicit-explicit}. Denote by $\sigma_1 \leq \sigma_2 \leq \ldots \leq \sigma_{N-1}$ the eigenvalues of $\bm{H}$ in \eqref{eq:tr_subproblem}. If $\sigma_1 \geq 0$, then subproblem \eqref{eq:tr_subproblem} is a convex problem, and thus the first-order convergence (Theorem~\ref{thm:converge_to_kkt}) naturally guarantees the global convergence of the iteration. Below, we assume $\sigma_1 < 0$ and choose $\bm{\xi}$ to be the unit-norm eigenvector of the matrix $\bm{H}$ corresponding to the smallest eigenvalue $\sigma_1$ (\textit{i.e.}, $\lVert \bm{\xi} \rVert_2 = 1$). We define $w^{(1)} := \bm{\xi}^{\top} \bm{w}$, for any $\bm{w} \in \mathbb{R}^{N-1}$. The proof of global convergence relies on the following proposition.

\begin{proposition}[Proposition 2.2~\cite{carmon2020first}]\label{propos:global_convergence}
Let $g^{(1)} \neq 0$. The global minimizer $\bm{d}_{*}$ of problem \eqref{eq:tr_subproblem} is the unique stationary point satisfying the condition $d^{(1)}_* g^{(1)} \leq 0$.
\end{proposition}

Building upon the above proposition, we further prove that the limit point $\bm{d}_*$ from Theorem~\ref{thm:converge_to_kkt} is the global minimizer of the subproblem \eqref{eq:tr_subproblem}.

\begin{theorem}\label{thm:global_optimality}
If $g^{(1)} \neq 0$ and the initial point is chosen as $\bm{d}_0 = -r \bm{g}/\|\bm{g}\|_2$, then there exists an $\eta_{\max} > 0$ such that for any step size $\eta \in (\eta_{\min}, \eta_{\max}]$, $\bm{d}_*$ is the global minimizer of the subproblem \eqref{eq:tr_subproblem}.
\end{theorem}
\begin{proof}
Theorem~\ref{thm:converge_to_kkt} establishes that $\bm{d}_{k} \to \bm{d}_{*}$ and that $\bm{d}_{*}$ is a KKT point of \eqref{eq:tr_subproblem}. By Proposition~\ref{propos:global_convergence}, we only need to prove that
\begin{equation}\label{negative_product_optimal_condition_tr}
d^{(1)}_* g^{(1)} \leq 0.
\end{equation}
We now use induction to establish that
\begin{equation}\label{Conclusions_of_the_inductive_approach_tr}
d^{(1)}_k g^{(1)} \leq 0 \quad \text{for all}\ k \geq 0,
\end{equation}
and thus taking the limit yields \eqref{negative_product_optimal_condition_tr}. By the choice of the initial point, we have
\begin{equation*}
d^{(1)}_0 g^{(1)} = -r (g^{(1)})^2 \leq 0.
\end{equation*}
Next, assuming $d^{(1)}_k g^{(1)} \leq 0$, we show $d^{(1)}_{k+1} g^{(1)} \leq 0$. We rewrite the iteration \eqref{eq:imex-scheme} as
\begin{equation}\label{semi_implict_in_global_convergence_tr}
\begin{aligned}
&[ \bm{I} + \eta ( \bm{D} + \lambda_{k+1} \bm{I} ) ] \bm{d}_{k+1} = \bm{d}_{k} - \eta ( \bm{T} \bm{d}_{k} + \bm{g} ) = \bm{d}_{k} - \eta ( ( \bm{H} - \bm{D} ) \bm{d}_{k} + \bm{g} )\\
= &( \bm{I} + \eta \bm{D} ) \bm{d}_{k} - \eta ( \bm{H} \bm{d}_{k} + \bm{g} ) = (\bm{I} + \eta \bm{D}) \bm{d}_{k} - \eta \nabla f(\bm{d}_{k})\\
= &(\bm{I} + \eta (\bm{D}+\lambda_{k+1}\bm{I})) \bm{d}_{k} - \eta (\lambda_{k+1} \bm{d}_{k} + \nabla f(\bm{d}_{k})).\\
\end{aligned}
\end{equation}
Then
\begin{equation}\label{iterate_format_in_convergence_tr}
\bm{d}_{k+1} = \bm{d}_{k} - \eta (\bm{I} + \eta (\bm{D} + \lambda_{k+1}\bm{I} ))^{-1} (\lambda_{k+1} \bm{d}_{k} + \nabla f(\bm{d}_{k})).
\end{equation}
Multiplying both sides of equation \eqref{semi_implict_in_global_convergence_tr} on the left by $\bm{\xi}$ yields
\begin{equation*}
(1 + \eta \lambda_{k+1}) d^{(1)}_{k+1} + \eta \langle \bm{\xi}, \bm{D} ( \bm{d}_{k+1} - \bm{d}_{k} ) \rangle = (1 - \eta \sigma_1 ) d^{(1)}_k - \eta g^{(1)}.
\end{equation*}
Multiplying both sides of the equation by $g^{(1)}$ yields
\begin{equation*}
(1 + \eta \lambda_{k+1} ) d^{(1)}_{k+1} g^{(1)} = (1 - \eta \sigma_1 ) d^{(1)}_k g^{(1)} - \eta \left( ( g^{(1)} )^2 + \langle \bm{\xi}, \bm{D} ( \bm{d}_{k+1} - \bm{d}_{k} ) \rangle g^{(1)} \right).
\end{equation*}
Since $d^{(1)}_k g^{(1)} \leq 0$, sufficient conditions for $d^{(1)}_{k+1} g^{(1)} \leq 0$ are
\begin{subequations}\label{sufficient_condition_for_optimal_tr}
\begin{equation}
    1 - \eta \sigma_1 \geq 0 \iff \eta \geq \frac{1}{\sigma_1}; \label{eq:sufficient_a}
\end{equation}
\begin{equation}\label{sufficient_b_tr}
( g^{(1)} )^2 + \langle \bm{\xi}, \bm{D} ( \bm{d}_{k+1} - \bm{d}_{k} ) \rangle g^{(1)} \geq 0.
\end{equation}
\end{subequations}
With $\bm{M}_k = \bm{D} [\bm{I} + \eta (\bm{D} + \lambda_{k+1} \bm{I})]^{-1}$, we analyze the conditions under which \eqref{sufficient_b_tr} holds. From the iteration \eqref{iterate_format_in_convergence_tr}, we have
\begin{equation}\label{eq19_tr}
\begin{aligned}
\| \bm{D} (\bm{d}_{k+1} - \bm{d}_{k} ) \|_2 =& \| \eta \bm{M}_k(\lambda_{k+1} \bm{d}_{k}+ \nabla f(\bm{d}_{k})) \|_2 \leq \eta \|\bm{M}_k\| \| \lambda_{k+1} \bm{d}_{k} + \nabla f(\bm{d}_{k}) \|_2\\
=& \frac{\eta \|\bm{D}\| \| \lambda_{k+1} \bm{d}_{k} + \nabla f(\bm{d}_{k}) \|_2 }{ 1 + \eta ( \|\bm{D}\| + \lambda_{k+1} ) } \leq  \frac{\eta \| \bm{D} \| (\lambda_{\max}\|\bm{d}_{k}\|_2 + \|\nabla f(\bm{d}_{k}) \|_2) }{ 1 + \eta \|\bm{D}\| }.
\end{aligned}
\end{equation}

There exists $\bar{\eta}>0$ such that, for all $\eta \leq \bar{\eta}$, we have
\begin{equation}\label{eq20_tr}
\frac{\eta \|\bm{D}\| (\lambda_{\max} \|\bm{d}_{k}\|_2 + \|\nabla f(\bm{d}_{k})\|_2) }{ 1 + \eta \|\bm{D}\| } < \frac{ |g^{(1)}| }{ 2 },\quad \text{for all}\ k \geq 0.
\end{equation}
Combining \eqref{eq19_tr}, \eqref{eq20_tr} and $\|\bm{\xi}\|_2 = 1$, we have
\begin{equation*}
\begin{aligned}
( g^{(1)} )^2 + \bm{\xi}^{\top}(\bm{D} ( \bm{d}_{k+1} - \bm{d}_{k} )) g^{(1)} \geq & ( g^{(1)} )^2 - \| \bm{\xi} \|_2 \| \bm{D} ( \bm{d}_{k+1} - \bm{d}_{k} ) \|_2 |g^{(1)}|\\
> & ( g^{(1)} )^2 - \frac{ | g^{(1)} |^2 }{ 2 } = \frac{ ( g^{(1)} )^2 }{ 2 } > 0.
\end{aligned}
\end{equation*}
Thus, combining the step size condition \eqref{eq:step_cond2} and \eqref{eq:sufficient_a}, if the step size $\eta$ satisfies the following condition
\begin{equation*}
\max\left\{\frac{1}{\sigma_1}, \eta_{\min}\right\} = \eta_{\min} \leq \eta \leq \eta_{\max} := \min \left\{ \frac{1}{\|\bm{T}\|+2a}, \bar{\eta} \right\},
\end{equation*}
it is guaranteed that $d^{(1)}_{k+1} g^{(1)} \leq 0$. Thus, by induction, \eqref{Conclusions_of_the_inductive_approach_tr}  holds. Let $k \to \infty$ in \eqref{Conclusions_of_the_inductive_approach_tr} and using Proposition~\ref{propos:global_convergence}, we conclude that $\bm{d}_{*}$ must be a global minimizer of problem \eqref{eq:tr_subproblem}.
\end{proof}

\begin{remark}\label{rmk:hard_case_strategy}
The assumption $g^{(1)}\neq 0$ excludes the exact hard case of the TR subproblem, where the gradient is orthogonal to the eigenvector corresponding to the smallest eigenvalue of $\bm{H}$, so that the induction argument in Theorem~\ref{thm:global_optimality} does not apply when $g^{(1)}=0$. In our implementation, whenever the TR subproblem is in the (nearly) hard case, we simply replace the deterministic initial point $\bm{d}_0=-r\bm{g}/\|\bm{g}\|_2$ in Algorithm~\ref{alg:adaptive-implicit-explicit} by a feasible point drawn from a continuous distribution whose support is $\{\bm{s}\mid\|\bm{s}\|_2\leq r\}$. As will be demonstrated in Subsection~\ref{subsec:beyond_theoretical_scope}, the IMEX-TR method with this randomized initialization still converges to SP-IIs in our numerical experiments.
\end{remark}

\section{Numerical Experiments}\label{sec:numerical}

This section presents a comprehensive numerical evaluation of the proposed IMEX-TR method. To systematically compare the proposed method with several representative first-order schemes, including gradient flow methods and a Bregman proximal gradient algorithm, we demonstrate that IMEX-TR consistently converges to SP-IIs, whereas conventional methods often stagnate at unstable SP-Is. And we show that IMEX-TR exhibits robustness to initialization and is capable of escaping saddle points to explore complex energy landscapes. Furthermore, we demonstrate that when the TR subproblem is in the nearly hard case, the IMEX-TR method can still converge to SP-IIs when the subproblem solver uses a random initial point.

Parameters used in Algorithm~\ref{alg:unified_imex_tr} are set to standard values in TR literature to balance convergence speed and numerical stability, $\gamma_C = 0.5$, $\gamma_E = 2$, $\gamma_\lambda = 1.5$, $\theta = 10^{-4}$, $\underline{\mu}=1$, $\bar{\mu}=10^{5}$, $\mu_0 = 1$, $r_0=1$, $\nu_0=5$ and $\varepsilon=10^{-10}$. In Algorithm~\ref{alg:adaptive-implicit-explicit}, the step size $\eta = 0.1$ is chosen empirically to ensure efficient inner-loop convergence with a tolerance $\varepsilon_{\text{sub}}=10^{-13}$.

\subsection{Convergence}
\label{subsec:sp2_convergence}
In this subsection, we compare the IMEX-TR method with representative methods, including the adaptive accelerated Bregman proximal gradient (AA-BPG)~\cite{jiang2020efficient}, and several gradient flow methods, including the semi-implicit scheme (SIS), first-order stabilized semi-implicit scheme (SSIS1), second-order stabilized semi-implicit scheme (SSIS2)~\cite{shen2010numerical}, the invariant energy quadratization (IEQ)~\cite{yang2016linear}, the scalar auxiliary variable (SAV)~\cite{shen2019new}, and exponential time differencing (ETD) methods~\cite{fu2022energy}.\footnote[1]{Since the results of SIS, SSIS1, and IEQ are similar to those of other methods except IMEX-TR, the corresponding results are not reported.} All algorithms start from the same initial condition to demonstrate that the IMEX-TR method, by effectively using Hessian information, converges to SP-IIs, while other methods may converge to SP-Is.

\begin{table}[htbp]
    \centering
    \caption{Model parameters, computational domains, and initial configurations for three numerical examples.}
    \label{tab:example_params}
    \begin{tabular}{c|c|c|c}
    \hline
    Example & Parameters $(\tau, \gamma)$ & Domain & Initial Phase \\
    \hline
    1 & $(-0.35, 0.7)$ & $[0, 2\sqrt{2}\pi]^3$ & LAM \\
    2 & $(-0.001, 0.4)$ & $[0, 2\sqrt{2}\pi]^3$ & LAM \\
    3 & $(-0.28, 0.32)$ & $[0, 2\sqrt{6}\pi]^3$ & HEX \\
    \hline
    \end{tabular}
\end{table}

We illustrate the performance of IMEX-TR through three different examples. Table~\ref{tab:example_params} summarizes the model parameters $(\tau, \gamma)$, computational domains, and initial configurations employed in the three numerical experiments. The initial ordered configurations---specifically, three-dimensional lamellar (LAM) and hexagonal (HEX) phases---are constructed via Fourier series expansion of the form~\cite{jiang2013discovery}
\[
  \psi(\bm{x}) = \sum_{\bm{k}\in\Xi }\hat{\psi}(\bm{k})e^{i(\bm{B}\bm{k})^{\top}\bm{x} },
\]
where the reciprocal lattice point sets $\Xi$, Fourier coefficients $\hat{\psi}(\bm{k})$, and reciprocal lattices $\bm{B}$ are detailed in Table~\ref{tab:initial_value}.

\begin{table}[htbp]
    \centering
    \caption{Initial reciprocal lattice point sets $\Xi$, nonzero Fourier coefficients $\hat{\psi}(\bm{k})$, and reciprocal lattice matrices $\bm{B}$ for the LAM and HEX phases.}
    \label{tab:initial_value}
    \begin{tabular}{c|c|c|c}
    \hline
    Phase & $\Xi$ & $\hat{\psi}(\bm{k})$ & $\bm{B}$ \\
    \hline
    LAM & $\{(1,0,0),(-1,0,0)\}$ & $0.3$ & $\frac{1}{\sqrt{2}}\bm{I}$ \\
    \hline
    HEX & \makecell{$\{(0,1,-1),(0,-1,1),(1,0,-1),$\\ $(-1,0,1),(1,-1,0),(-1,1,0)\}$} & $0.3$ & $\frac{1}{\sqrt{6}}\bm{I}$ \\
    \hline
\end{tabular}
\end{table}

\begin{figure}[htbp]
  \centering
  \includegraphics[width=.9\linewidth]{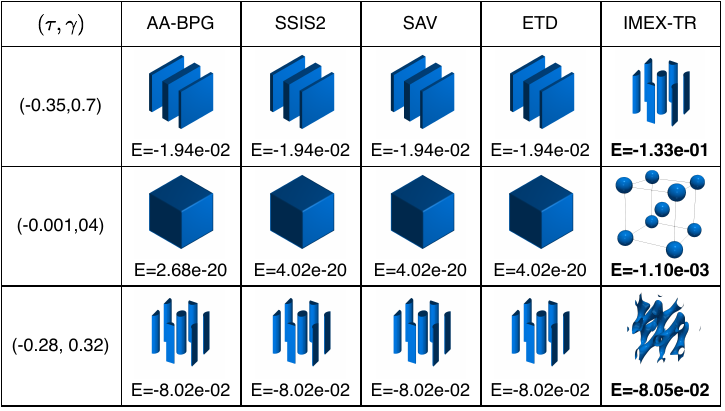}
  \caption{Final physical states and energy values obtained by different methods. Across the three examples, IMEX-TR converges to SP-IIs, including HEX, BCC, and a metastable network, consistently achieving lower energy states compared to standard first-order schemes}
  \label{fig:diff_method_phases}
\end{figure}

\begin{figure}[htbp]
  \centering
  \includegraphics[width=.8\linewidth]{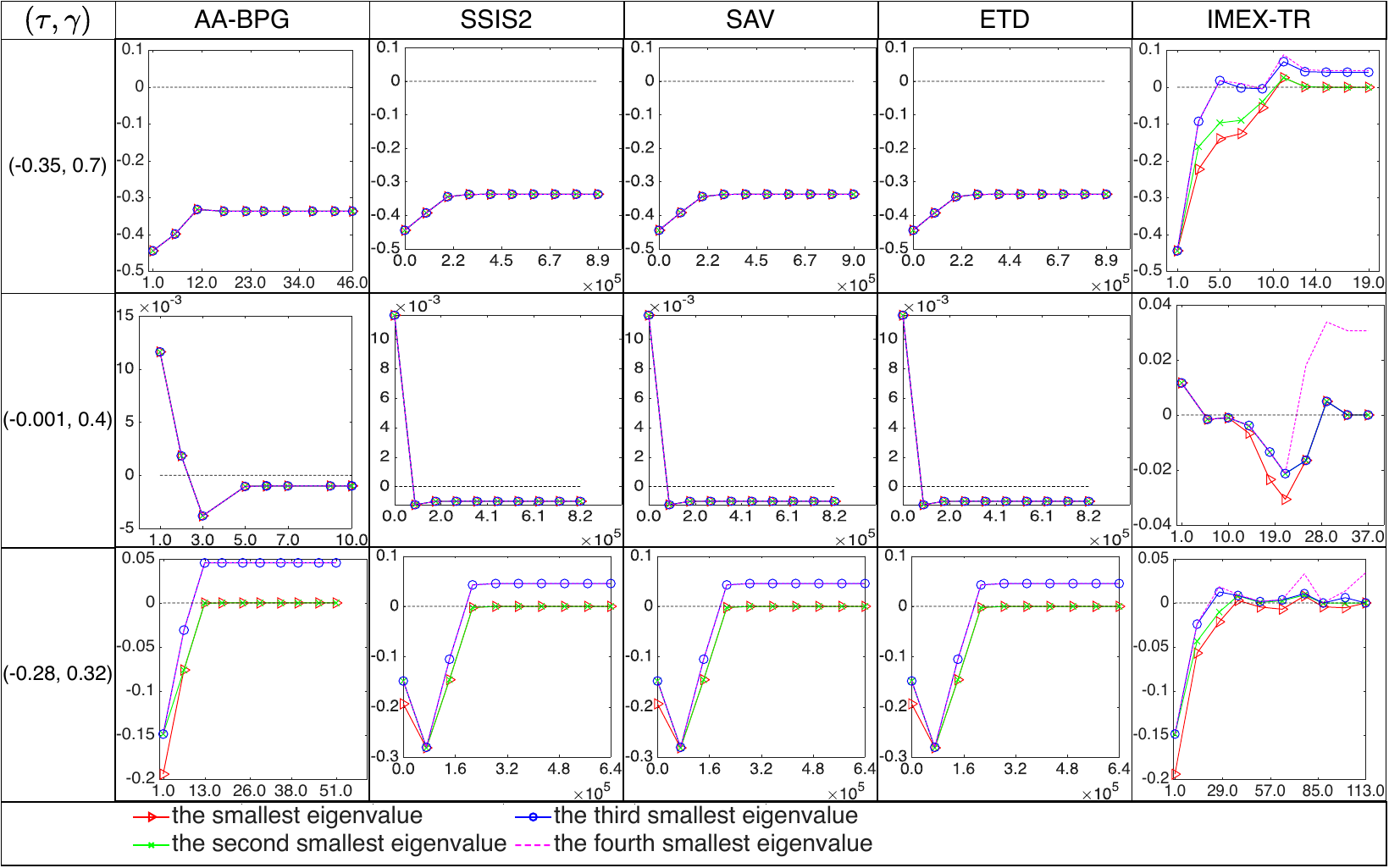}
  \caption{Evolution of the four smallest eigenvalues of the Hessian matrix. The dashed line indicates the zero threshold. IMEX-TR consistently reaches the SP-II region (non-negative eigenvalues), whereas first-order methods often stagnate at SDPs with negative curvature.}
  \label{fig:diff_method_eig}
\end{figure}

Figure~\ref{fig:diff_method_phases} presents the converged stationary points and their corresponding energy values. In Examples 1 and 2, conventional first-order schemes fail to escape the initial energy basins and stagnate at unstable stationary points. Specifically, in Example 1, first-order methods remain trapped in the LAM state, while in Example 2, they converge to a disordered state (\textit{i.e.}, $\phi(\bm{x})\approx 0$). In contrast, the IMEX-TR method successfully converges to SP-IIs, namely the HEX and body-centered cubic (BCC) phases, both exhibiting lower free energies. In Example 3, all methods reach SP-IIs; however, IMEX-TR moves away from the HEX basin and converges to a lower-energy network SP-II with energy $E=-8.05\times 10^{-2}$, compared with the HEX phase energy $E=-8.02\times 10^{-2}$.

To further verify the stability of these stationary points, we perform an eigenvalue analysis of the Hessian matrix. Figure~\ref{fig:diff_method_eig} illustrates the evolution of the four smallest eigenvalues throughout the iterative process. For the first-order methods in Examples 1 and 2, the persistence of significant negative eigenvalues confirms that the states reached are SDPs. Conversely, the IMEX-TR method converges to a point where the Hessian eigenvalues are non-negative, which satisfies the SP-II condition.

\subsection{Robustness to initial conditions}

\begin{figure}[htbp]
  \centering
  \includegraphics[width=.8\linewidth]{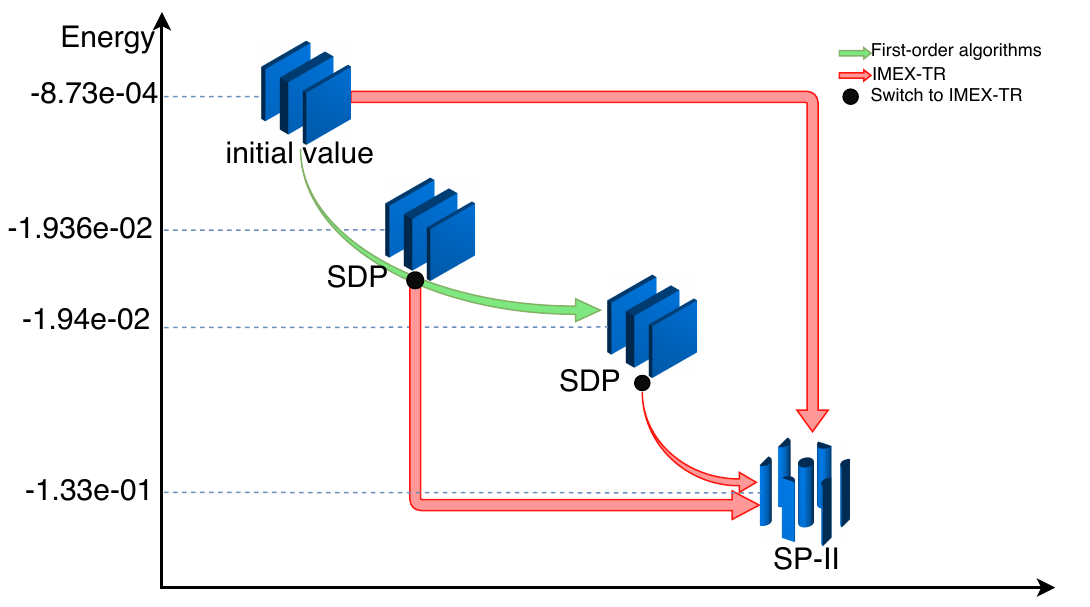}
  \caption{Trajectories of IMEX-TR activated at different stages of first-order methods. Whether started from the initial point, an intermediate state, or the converged saddle point, IMEX-TR consistently converges to the same (meta-)stable SP-II (HEX phase), demonstrating its efficacy in exploiting negative curvature directions to escape unstable stationary points.}
  \label{fig:mix_apg_tr_lam2hex}
\end{figure}

We further evaluate the robustness of IMEX-TR by activating it at distinct stages along the trajectory of first-order methods. As Figure~\ref{fig:mix_apg_tr_lam2hex} illustrates, IMEX-TR launches on three representative states: (i) the initial configuration, (ii) an intermediate state, and (iii) the converged SDP stage attained by first-order solvers. The behavior of IMEX-TR remains similar regardless of the specific first-order method used.

These results show that IMEX-TR consistently converges to an SP-II, a fundamental advantage arising from its use of Hessian information to capture the curvature of the energy landscape. This capability allows the algorithm to identify and exploit descent directions associated with negative curvature, thereby enabling it to divert the optimization trajectory away from SP-Is during the early stages or to escape them after stagnation.
Consequently, the IMEX-TR reliably avoids SP-Is and locates an SP-II, \textit{i.e.}, the physical (meta-)stable phase, regardless of whether it is deployed from the beginning or as a post-processing step for other solvers.

\subsection{Beyond the theoretical scope}\label{subsec:beyond_theoretical_scope}

In practical computations, the TR subproblem may become nearly hard, i.e., $|g^{(1)}|$ is extremely small, so that the assumption $g^{(1)}\neq 0$ of Theorem~\ref{thm:global_optimality} is only marginally satisfied. As described in Remark~\ref{rmk:hard_case_strategy}, whenever such subproblems are encountered, we initialize the inner solver with a feasible random value. In this subsection, we verify numerically that IMEX-TR equipped with this randomized initialization strategy still converges to SP-IIs.

For $(\tau,\gamma)=(-0.09,0.34)$, with the computational domain, initial LAM configuration, and all other settings chosen as in Example~1 of Table~\ref{tab:example_params}, the algorithm encounters nearly hard TR subproblems satisfying $|g^{(1)}|\leq 10^{-12}$.
To isolate the effect of the inner initialization, we keep all algorithmic settings unchanged and compare the deterministic initial point used in Theorem~\ref{thm:global_optimality} with a feasible random value generated from an continuous distribution whos support is $\{\bm{s}\mid\|\bm{s}\|_2\leq r\}$, where $r$ is the trust region radius.
As shown in Figure~\ref{fig:beyond_theoretical_scope_left}, the deterministic initialization leads to a BCC phase with energy $E=-1.03\times10^{-2}$.
Although the displayed Hessian eigenvalues change sign during the iteration, several of them eventually remain negative.
The Hessian is therefore not positive semidefinite, and the final state is not an SP-II. In Figure~\ref{fig:beyond_theoretical_scope_right}, randomized initialization leads instead to a HEX phase with the lower energy $E=-1.32\times10^{-2}$. At termination, all the displayed smallest Hessian eigenvalues are nonnegative. The final state therefore satisfies the SP-II condition.
Since the two runs differ only in the inner initialization used for the nearly hard subproblems, this experiment confirms that the randomized initialization strategy enables IMEX-TR to converge to an SP-II even when nearly hard subproblems are encountered, in line with the findings in~\cite{beck2018globally} regarding random initialization.

\begin{figure}[htbp]
    \subfigure[]{
        \includegraphics[width=.45\linewidth]{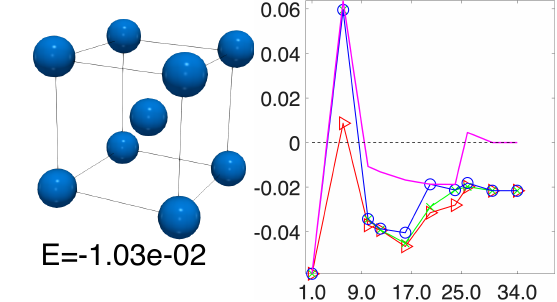}
        \label{fig:beyond_theoretical_scope_left}
    }
    \subfigure[]{
        \includegraphics[width=.45\linewidth]{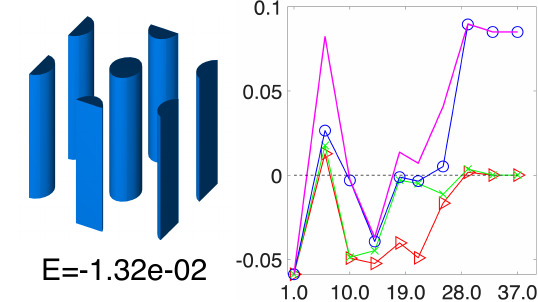}
        \label{fig:beyond_theoretical_scope_right}
    }
  \caption{Numerical behavior of IMEX-TR for the same example beyond the deterministic theoretical scope, with parameters $(\tau,\gamma)=(-0.09,0.34)$ and other settings as in Example~1 of Table~\ref{tab:example_params}. The only difference between (a) and (b) is the initial point used by the inner solver in the nearly hard case subproblems ($|g^{(1)}|\leq 10^{-12}$). In (a), the deterministic initial point from Theorem~\ref{thm:global_optimality} leads to a final state that is not an SP-II, whereas in (b), a randomized initial point leads to convergence to an SP-II.}
  \label{fig:beyond_theoretical_scope}
\end{figure}

\section{Conclusion}\label{sec:conclusion}

We have demonstrated that a systematic focus on SP-IIs is essential for reliably computing stable and metastable states of the LB model.
A new algorithm, IMEX-TR, was developed to compute such points with $\mathcal{O}(N\log N)$ complexity per iteration for the subproblem, exploiting the diagonal structure of the Hessian in reciprocal and physical spaces.
The method is provably convergent to SP-IIs under the stated assumptions and consistently avoids unstable stationary points in the numerical experiments, including cases where first-order methods stagnate at saddles or disordered states.
Finally, although the IMEX-TR focused on the LB model, the central idea of targeting SP-IIs is directly applicable to other Landau-type models of similar structure and various spatial discretization schemes.

\section*{Acknowledgments}

\section*{Funding}
The first author was supported by the NSFC (12271291) and New Cornerstone Investigator Program (100001127). The third author was supported by the National Key R\&D Program of China (2023YFA1008802), the NSFC (12288101), the Science and Technology Innovation Program of Hunan Province (2024RC1052). The fourth author was supported by the National Key R\&D Program of China (2023YFB3001604). This work is also partially supported by the High Performance Computing Platform of Xiangtan University.

\bibliographystyle{plain}
\bibliography{refs}


\end{document}